\documentclass[11pt]{article}
\usepackage[margin=0.75in]{geometry}

\usepackage{amssymb,amsmath,mathtools,amsthm,thmtools,amsfonts}
\usepackage{thm-restate}

\usepackage{url} 
\usepackage{hyperref}
\hypersetup{
colorlinks,
linkcolor={red}, 
citecolor={black},
urlcolor={blue!60!black},
pdftitle={Fractional Clique Decompositions of Dense Hypergraphs},
pdfauthor={Michelle Delcourt, Thomas Lesgourgues, Luke Postle}}

\usepackage[utf8]{inputenc}
\usepackage{color}

\usepackage{enumitem}
\setlist[itemize]{topsep=.1in,itemsep=.1in}
\setlist[enumerate]{topsep=.1in,itemsep=.1in}
\setlist[enumerate,1]{label={(\roman*)}}
\setlist[enumerate,2]{label={(\alph*)}}
\setlist[enumerate,3]{label={(\arabic*)}}

\setlist[itemize]{nolistsep,noitemsep, topsep=0pt}

\newcommand{\mc}[1]{\mathcal{#1}}%

\newcommand{\Prob}[1]{\ensuremath{\mathbb P\left[#1\right]}}

\DeclarePairedDelimiter{\brackets}{[}{]}

\usepackage[noabbrev,capitalise]{cleveref}

\theoremstyle{plain}
\newtheorem{thm}{Theorem}[section]
\newtheorem{lem}[thm]{Lemma}
\newtheorem{proposition}[thm]{Proposition}

\newtheorem{cor}[thm]{Corollary}

\newtheorem{conj}[thm]{Conjecture}

\newenvironment{proofclaim}[1][\proofname]
{\proof[#1]}
{\endproof}

\declaretheorem[
  style=plain,
  name=Claim,
  within=theorem,
]{claim}

\newenvironment{nonlateproof}[1]
 {%
  \renewcommand{\theclaim}{\ref{#1}.\arabic{claim}}%
  \begin{proof}%
 }
 {\end{proof}}

\usepackage{etoolbox}
 \AtEndEnvironment{proof}{\setcounter{claim}{0}}

\theoremstyle{plain} 
\newcommand{\thistheoremname}{}
\newtheorem{genericthm}{\thistheoremname}

\theoremstyle{definition}
\newtheorem{definition}[thm]{Definition}

\crefname{equation}{}{}
\crefname{lem}{Lemma}{Lemmas}
\crefname{claim}{Claim}{Claims}
\crefname{thm}{Theorem}{Theorems}
\crefname{enumi}{}{}

\usepackage[dvipsnames, table]{xcolor}

\title{Fractional Clique Decompositions of Dense Hypergraphs}

\author{
Michelle Delcourt
\thanks{Department of Mathematics, Toronto Metropolitan University (formerly named Ryerson University),
Toronto, Ontario M5B 2K3, Canada {\tt mdelcourt@torontomu.ca}. Research supported by NSERC under Discovery Grant No. 2019-04269 and a Sloan Research Fellowship.} \and 
Thomas Lesgourgues
\thanks{Combinatorics and Optimization Department,
University of Waterloo, Waterloo, Ontario N2L 3G1, Canada {\tt tlesgourgues@uwaterloo.ca}.} \and
Luke Postle
\thanks{Combinatorics and Optimization Department,
University of Waterloo, Waterloo, Ontario N2L 3G1, Canada {\tt lpostle@uwaterloo.ca}. Partially supported by NSERC
under Discovery Grant No. 2019-04304.}}
\date{\today}

\begin{document}

\maketitle
\begin{abstract}
In 2014, Keevash famously proved the existence of $(n,q,r)$-Steiner systems as part of settling the Existence Conjecture of Combinatorial Designs (dating from the mid-1800s). In 2020, Glock, K\"uhn, and Osthus conjectured a minimum degree generalization: specifically that minimum $(r-1)$-degree at least $(1-\frac{C}{q^{r-1}})n$ suffices to guarantee that every sufficiently large $K_q^r$-divisible $r$-uniform hypergraph on $n$ vertices admits a $K_q^r$-decomposition (where $C$ is a constant that is allowed to depend on $r$ but not $q$). 

The best-known progress on this conjecture is from the second proof of the Existence Conjecture by Glock, K\"uhn, Lo, and Osthus in 2016 who showed that $(1-\frac{C}{q^{2r}})n$ suffices. The fractional relaxation of the conjecture is crucial to improving the bound; for that, only the slightly better bound of $(1-\frac{C}{q^{2r-1}})n$ was known due to Barber, K\"uhn, Lo, Montgomery, and Osthus from 2017. 

Our main result is to prove that $(1-\frac{C}{q^{r-1+o(1)}})n$ suffices for the fractional relaxation. Combined with the work of R{\"o}dl, Schacht, Siggers, and Tokushige from 2007, this also shows that such hypergraphs admit approximate $K_q^r$-decompositions. 

\end{abstract}


\section{Introduction}

An \emph{$(n,q,r)$-Steiner system} for integers $n > q > r \geq 1$ is a set $S$ of $q$-subsets of an $n$-set $X$ such that each $r$-subset of $X$ is contained in exactly one element of $S$. From the hypergraph-theoretic perspective, an $(n,q,r)$-Steiner system is equivalent to a decomposition of the edges of $K_n^r$ (the complete $r$-uniform hypergraph on $n$ vertices) into edge-disjoint copies of $K_q^r$ (referred to as cliques). The necessary divisibility conditions for the existence of an $(n, q, r)$-Steiner system are the following: for each $0 \leq i \leq r-1,$ we require $\binom{q-i}{r-i}~|~~\binom{n-i}{r-i}$.

Whether these divisibility conditions suffice for all large enough $n$ was a central open case of a notorious folklore conjecture from the 1800s called the Existence Conjecture of Combinatorial Designs. (More generally, the Existence Conjecture posited the existence of an \emph{$(n,q,r,\lambda)$-design} where each $r$-subset is in $\lambda$ elements of $S$ for some $\lambda\ge 1$; that said, the $\lambda=1$ case was viewed as the hardest case by far.) In 1847, in one of the oldest theorems of design theory, Kirkman~\cite{K47} proved the existence of $(n,3,2)$-Steiner systems. In the 1970s, Wilson~\cite{W72-EC1, W72-EC2, W75-EC3} revolutionized design theory when he proved the Existence Conjecture for $r=2$ (the graph case). In 2014, Keevash ushered in a new era of design theory when he proved the Existence Conjecture in full. 

A central open question in design theory (and arguably \emph{the} central open question in extremal design theory) is Nash-Williams' Conjecture on the minimum degree threshold needed to guarantee that a graph admits a $K_3$-decomposition (provided it satisfies the necessary divisibility conditions). Nash-Williams' Conjecture and its generalizations to $K_q$-decompositions and $K_q^r$-decompositions are considered fundamental in the area; indeed they also imply the related \emph{completion problem} for Steiner systems, namely how sparse a \emph{partial} $(n,q,r)$-Steiner system has to be to guarantee that it extends to an $(n,q,r)$-Steiner system. These conjectures can be viewed as generalizations to higher uniformities of foundational results in extremal graph theory: namely of the edge density threshold for the existence of one triangle (Mantel~\cite{mantel} in 1907) or one $K_q$ (Tur\'an~\cite{turan1941extremal} in 1941) for the $0$-uniform case; whereas for the $1$-uniform setting, the central results are the minimum degree threshold of a triangle factor (Corr\'adi and Hajnal~\cite{CH63} in 1963) and a $K_q$-factor (Hajnal and Szemer\'edi~\cite{HS70} in 1970) and the minimum $(r-1)$-degree threshold for matchings (Dirac~\cite{dirac1952some} in 1952) and hypergraph matchings (R\"odl, Ruci\'nski, and Szemer\'edi in 2006~\cite{RRS06} for $r=3$ and 2008~\cite{RRS08} for all $r$).

As to the conjectures, an $r$-uniform hypergraph ($r$-graph for short) $G$ is \emph{$K_q^r$-divisible} if for all $i\in \{0,1,\ldots, r-1\}$, we have that $\binom{q-i}{r-i}~|~|\{e\in G: S\subseteq e\}|$ for all subsets $S$ of $V(G)$ of size $i$. Nash-Williams' Conjecture states that every sufficiently large $K_3$-divisible graph $G$ on $n$ vertices with minimum degree $\delta(G)\ge 3n/4$ admits a $K_3$-decomposition. The (false) folklore generalization of Nash-Williams' conjectured that every sufficiently large $K_q$-divisible graph $G$ on $n$ vertices with $\delta(G)\ge (1-\frac{1}{q+1})n$ admits a $K_q$-decomposition. For hypergraphs, it is natural to consider the \emph{minimum $(r-1)$-degree} of an $r$-graph $G$, denoted $\delta_{r-1}(G)$, which is the minimum over all $(r-1)$-sets $S$ of $G$ of the number of edges of $G$ containing $S$. Glock, K\"uhn, and Osthus~\cite{GKO20Survey} conjectured the following high minimum degree generalization of the Existence Conjecture as follows.

\begin{conj}[Glock, K\"uhn, Osthus~\cite{GKO20Survey} 2020]\label{conj:GKO}
For each integer $r\ge 3$, there exists a constant $C > 0$ such that the following holds for all integers $q > r$ and $n$ sufficiently large: If $G$ is a $K_q^r$-divisible graph on $n$ vertices with $\delta_{r-1}(G)\ge \left(1-\frac{C}{q^{r-1}}\right)n$, then $G$ admits a $K_q^r$-decomposition.
\end{conj}

Crucial to the known results and proofs for the conjectures mentioned above is the notion of a \emph{fractional $K_q^r$-decomposition} of an $r$-graph $G$ which is an assignment of non-negative weights to each copy of $K_q^r$ in $G$ such that for each $e\in E(G)$, the sum of the weights of cliques  containing $e$ is exactly $1$. Note the existence of a $K_q^r$-decomposition necessitates the existence of a fractional $K_q^r$-decomposition. Conversely, the seminal work of Haxell and R\"odl~\cite{HR01} from 2001 using Szemer\'edi's Regularity Lemma and R\"odl's nibble method shows that the existence of a fractional $K_q$-decomposition implies the existence of an \emph{approximate} $K_q$-decomposition (that is a \emph{packing} of edge-disjoint $K_q$'s covering a $1-o(1)$ proportion of the edges). More generally, they showed the maximum weight of a fractional $K_q$-packing is approximately equal to the maximum weight of a $K_q$-packing. Their work was generalized to hypergraphs by R{\"o}dl, Schacht, Siggers, and Tokushige~\cite{RSST07} in 2007 using the hypergraph regularity lemma. Indeed, the known decomposition results then use variations of the absorption method to transform approximate decompositions into full decompositions.

To simplify discussion of the known results, it is useful to define the \emph{decomposition threshold of $K_q^r$}, denoted by $\delta_{K_q^r}$, as $\lim\sup_{n\rightarrow \infty} \delta_{K_q^r}(n)/n$ where $\delta_{K_q^r}(n)$ is the smallest integer $d$ such that every $K_q^r$-divisible graph $G$ on $n$ vertices with $\delta_{r-1}(G)\ge d$ has a $K_q^r$-decomposition. Similarly, it is useful to define the \emph{fractional $K_q^r$-decomposition threshold} $\delta^*_{K_q^r}$ as the infimum of all real numbers $c$ such that every $r$-graph $G$ with $\delta_{r-1}(G)\ge c\cdot v(G)$ admits a fractional $K_q^r$-decomposition. 

In a breakthrough result from 2014 using iterative absorption, Barber, K\"uhn, Lo and Osthus showed that $\delta_{K_3}=\delta^*_{K_3}$, thus reducing the asymptotic version of Nash-Williams' Conjecture to its fractional relaxation (a construction of Graham from 1970 shows that Nash-Williams' would be tight and even that $\delta^*_{K_3}\ge 3/4$). The best-known fractional bound is by Delcourt and Postle~\cite{DP2021progress} who showed that $\delta^*_{K_3} \leq\frac{7+\sqrt{21}}{14}<0.82733$, which improved upon earlier progress by Garaschuk~\cite{garaschuk2014linear}, Dross~\cite{dross2016fractional}, and Dukes and Horsley~\cite{dukes_minimum_2020}.

As for $K_q$-decompositions for $q\ge 4$, the best-known decomposition result is by Glock, K\"{u}hn, Lo, Montgomery, Osthus~\cite{GKLMO19} in 2019 who showed that $\delta_{K_q}=\delta^*_{K_q}$ (a construction of Gustavsson~\cite{gustavsson1991decompositions} from 1991 shows that $\delta^*_{K_q}\ge 1-\frac{1}{q+1}$). Surprisingly, the authors and Henderson~\cite{DHLP25Counter} recently proved that the folklore generalization is false for every $q\geq 4$. They proved that for every integer $q\geq 4$, there exists a constant $c>1$ such that $\delta^*_{K_q}\geq 1 - \frac{1}{c\cdot(q+1)}$ (and even that $c$ can be taken to be $\frac{\sqrt{2}+1}{2}-\varepsilon$ for any $\varepsilon>0$ provided $q$ is large enough). The current best-known upper bound on $\delta^*_{K_q}$, and the first and only result to yield a bound whose denominator is linear in $q$, is due to Montgomery~\cite{montgomery2019fractional} who proved that $\delta^*_{K_q}\leq1-\frac{1}{100q}$, after incremental progress by various authors, notably~\cite{BKLMO2017fractional,D12,Y05Fractional}. 

As for $K_q^r$-decompositions for $r\ge 3$, even the order of the polynomial of $q$ in the denominator for fixed $r$ is still open. Keevash's original proof of the Existence Conjecture shows that $\delta^*_{K_q^r} \le 1-\varepsilon$ for some $\varepsilon > 0$ but that $\varepsilon$ is exponentially small in $q$ (namely at most $2^{-q^8}$). The first polynomial bound dates from the second proof of the Existence Conjecture by Glock, K\"uhn, Lo, and Osthus~\cite{GKLO16} from 2016 as follows.

\begin{thm}[Glock, K\"uhn, Lo, and Osthus~\cite{GKLO16} 2016]\label{thm:GKLO}
For each integer $r\ge 3$, there exists a constant $C > 0$ such that the following holds for all integers $q > r$ and $n$ sufficiently large: If $G$ is a $K_q^r$-divisible graph on $n$ vertices with $\delta_{r-1}(G)\ge \left(1-\frac{C}{q^{2r}}\right)n$, then $G$ admits a $K_q^r$-decomposition.   
\end{thm}

Meanwhile, the fractional decomposition threshold $\delta^*_{K_q^r}$ has been studied in its own right. Improving upon earlier results from 2007 by Yuster~\cite{Y07} and 2012 by Dukes~\cite{D12,D15corrigendum}, the best-known bound is by Barber, K\"{u}hn, Lo, Montgomery, and Osthus~\cite{BKLMO2017fractional} from 2017 proved the following.

\begin{thm}[Barber, K\"{u}hn, Lo, Montgomery, and Osthus~\cite{BKLMO2017fractional} 2017]
For each integer $r\ge 3$, there exists a constant $C > 0$ such that the following holds for every integer $q > r$:  
$$\delta^*_{K_q^r}\leq 1-\frac{C}{q^{2r-1}}.$$
\end{thm}

Note that both the decomposition and fractional decomposition result differ from the conjecture by a factor of two in the exponent. 

The main result of this paper is a resolution of Conjecture~\ref{conj:GKO} up to an additive factor of $\varepsilon$ in the exponent as follows.

\begin{thm}\label{thm:Main}
For every integer $r \geq 3$ and real $\varepsilon\in (0,1]$, there exists a constant $C > 0$ such that the following holds for every integer $q > r$: 
$$\delta^*_{K_q^r}\leq 1-\frac{C}{q^{r-1+\varepsilon}}.$$
\end{thm}

Combined with the work of R{\"o}dl, Schacht, Siggers, and Tokushige~\cite{RSST07} this yields the following corollary.

\begin{cor}\label{cor:Main}
For every integer $r \geq 3$ and real $\varepsilon\in (0,1]$, there exists a constant $C > 0$ such that the following holds for every integer $q > r$ and $n$ large enough: If $G$ is an $r$-graph with $\delta_{r-1}(G)\ge 1-\frac{C}{q^{r-1+\varepsilon}}$, then $G$ admits a $K_q^r$-packing covering $(1-\varepsilon)\cdot e(G)$ edges of $G$. 
\end{cor}

We remark that for lower bounds for Conjecture~\ref{conj:GKO}, Glock, K\"{u}hn, and Osthus observed that one can adapt a construction of Kostochka, Mubayi and Verstra\"{e}te~\cite{KMV14}, itself a blow-up of one by R{\"o}dl and {\v{S}}inajov{\'a}~\cite{RS94}, to prove that for all $q,r,n_0$, there exists a $K_q^r$-free $r$-uniform hypergraph $G$ on $n\geq n_0$ vertices, with $\delta(G)\geq(1-c_r\frac{\log q}{q^{r-1}})n$ and containing not a single copy of $K_q^r$. This trivially implies that $\delta_{K_q^r}^*\geq 1-\Theta_r(\frac{\log q}{q^{r-1}})$. In 2018, Lo and Zhao~\cite{LZ2018codegree}, proved that this bound is asymptotically tight for the existence of one copy of $K_q^r$. It would obviously be interesting to remove that $\log$ factor in the lower bound for $\delta^*_{K_q^r}$, as conjectured by Glock, K\"{u}hn, and Osthus.

\subsection{On Montgomery's Proof and the Difficulties of Generalizing to Hypergraphs}

As to our proof, it is useful to understand Montgomery's proof strategy~\cite{montgomery2019fractional} for showing $\delta^*_{K_q}\leq1-\frac{1}{100q}$ for graphs and the difficulties in generalizing this to hypergraphs. Essentially none of Montgomery's proof generalizes immediately to hypergraphs (which is presumably why there has been no subsequent progress on improving $\delta^*_{K_q^r}$) and indeed, even how to generalize the strategy itself is unclear. That said, we were able to develop a simplified version of Montgomery's proof that does generalize to hypergraphs when aided by the addition of a couple innovative ideas. In fact, our proof of Theorem~\ref{thm:Main} also holds for $r=2$ and so provides a conceptually streamlined proof of the weaker result that $\delta^*_{K_q} \le 1 - \frac{1}{q^{1+o(1)}}$ (the best-known bound previous to Montgomery's work was $1-O\left(\frac{1}{q^{3/2}}\right)$ by Barber, K\"uhn, Lo, Montgomery and Osthus~\cite{BKLMO2017fractional}). 

There are essentially three core parts of Montgomery's proof strategy: (1) show that one can convert an `almost' fractional decomposition into a fractional decomposition; (2) show that an appropriately sized complete graph less a matching (or the union of some matchings) admits a fractional decomposition; and (3) prove that a certain complicated random process yields an almost fractional decomposition into subgraphs that are cliques minus matchings. To finish, one concatenates the above steps in the appropriate order.

For (1), Montgomery proved that a $\frac{1}{2q+1}$-almost fractional $K_{2q+2}$-decomposition implies a fractional $K_q$-decomposition (where $\varepsilon$-almost means every edge receives weight in $[1-\varepsilon,1]$ rather than exactly $1$ --- see Definition~\ref{def:AlmostDecompo} for a formal definition). For (2), Montgomery proved that $K_{2q+2}-M$ has a fractional $K_q$-decomposition for any matching $M$; by induction this implied that $K_{32q+62}-\bigcup_{i\in [5]}M_i$ admits a fractional $K_q$-decomposition for any five matchings $M_1,\ldots, M_5$. It also implied (1) via a coloring/averaging trick which we will also utilize. We note that proving that $K_{2q+2}-M$ has a fractional $K_q$-decomposition is especially natural as it is essentially the first non-trivial special case of the folklore conjecture. (Indeed, the authors with Henderson originally found counterexamples to the folklore conjecture by considering the special case of deleting two matchings.) Step (3), which is by far the most complicated, then showed an almost fractional decomposition into such subgraphs. 

As to the difficulties in generalizing the above to hypergraphs, essentially no step immediately generalizes to hypergraphs. For the crucial step (2), Montgomery proved the existence of a $K_q$-decomposition by directly providing a set of weights for the various types of cliques so as to satisfy a certain linear system of equations. To invoke this approach for hypergraphs seems much too complicated; indeed, even the case of a perfect matching which is trivial for graphs by symmetry seems hard for hypergraphs using the weight-based approach. 

Another confounding issue is whether one should be trying to prove that a) deleting a hypergraph matching (a set of vertex-disjoint edges) admits a fractional $K_q^r$-decomposition, or b) deleting an $(n,r,r-1)$-Steiner system does. The latter is actually a special case of Conjecture~\ref{conj:GKO} (for $n=\Theta(q^{r-1})$ say) and so would seem more natural; that said, the former is the approach that works nicely with induction.

As for step (3), it is unclear how to generalize Montgomery's random process (let alone how to analyze it) since the process crucially uses selecting non-neighbors of vertices.

\subsection{On the Novelty of Our Proof}

\emph{How then did we overcome these difficulties?} The first key is to prove a version of step (1) \emph{before proving the deleting matching lemma of step (2)} so that we can invoke the resulting fixing lemma in the proof of step (2). To that end, we prove (Theorem~\ref{thm:MissingOneEdge}) that $K_{qr}^r-e$ (missing one edge) admits a fractional $K_q$-decomposition via a direct inclusion-exclusion calculation of clique weights. This allows us to prove a fixing lemma (Corollary~\ref{cor:ColorTrick}) using Montgomery's coloring trick and hence to convert a $\frac{1}{\binom{rq}{r}}$-almost fractional $K_{rq}^r$-decomposition into a fractional $K_q^r$-decomposition (Corollary~\ref{cor:AlmostToFull}). 

The second key is to prove that $K_{Cq}-M$ admits an almost fractional decomposition and then invoke Corollary~\ref{cor:AlmostToFull} to yield a fractional $K_q$-decomposition (see Theorem~\ref{thm:OneMatchingMissing}).  Crucially $M$ is a hypergraph matching here, both for this proof step but also for the purposes of induction to show a similar statement for the union of a constant number of matchings (Corollary~\ref{cor:MissingManyMatchings}). \emph{But how to prove $K_{Cq}-M$ admits an almost fractional $K_q$-decomposition?} If only one could sample the vertices with some constant probability $p$ but such that no edge of $M$ is chosen! While this cannot be done independently, there does exist a \emph{quasi-independent} probability distribution $\phi$ on the subsets of $e\in M$ with $\phi(e)=0$ but having the property that for any other subset $T\subsetneq e$, the probability its vertices are chosen is $p^{|T|}$ (provided the necessary condition that $p\le 1/2$ holds of course -- see Lemma~\ref{lem:ProbDist} for the construction). Then applying this distribution independently to each edge of $M$ (and selecting vertices not in $M$ independently with probability $p$) yields a distribution into cliques with each edge having the same probability of $p^r$. Finally, we have to throw out all samples with fewer than $q$ vertices but this happens with small probability as the expected number of vertices is much larger than $q$. 

For step (3), the key is to create a probability distribution over subgraphs $H$ of the form $K_{Cq}-M^*$ where $M^*$ is the union of at most $m$ matchings and where crucially every edge has approximately the same probability to be in $H$ (since such is equivalent to the desired almost decomposition) -- see Theorem~\ref{thm:AlmostBySampling}. To create this, we randomly sample $Cq$ vertices and keep the subgraphs which are missing at most $m$ matchings. While simple, the analysis to prove this happens with high probability requires some care. If one was missing $(n,r,r-1)$-Steiner systems, this analysis would be almost immediate but for hypergraph matchings this requires some clever accounting. In particular, we have to understand the structure of hypergraphs without isolated vertices -- namely that they admit a natural \emph{edge exploration ordering} (see Proposition~\ref{prop:Exploration}) that can be used to bound the probability of missing too many matchings.

\subsection{Organization of the Paper} 
In~\cref{sec:prelim}, we prove Theorem~\ref{thm:MissingOneEdge} and Corollaries~\ref{cor:ColorTrick} and~\ref{cor:AlmostToFull}. Before proving those results however, we prove a useful ``Concatenation Lemma'' (Lemma~\ref{prop:AlmostConcatenates}) for concatenating almost fractional decompositions. The ``missing matchings'' case is handled in~\cref{sec:MissingMatchings}, where we prove Theorems~\ref{thm:OneMatchingMissing} and Corollary~\ref{cor:MissingManyMatchings}. \cref{sec:sampling} is then dedicated to proving Theorem~\ref{thm:AlmostBySampling} and Proposition~\ref{prop:Exploration}. Finally, in~\cref{sec:MainProof} we concatenate the two decompositions from Theorem~\ref{thm:AlmostBySampling} and Corollary~\ref{cor:MissingManyMatchings} to prove our main result, Theorem~\ref{thm:Main}.

\section{Missing One Edge and a Fixing Lemma}\label{sec:prelim}

We denote by $[r-1]_0$ the set $\{0,1,\ldots,r-1\}$. For an $r$-graph $G$, that is, an $r$-uniform hypergraph $G$, a {\em fractional $F$-packing} of $G$ is an assignment of non-negative weights to each copy of $F$ in $G$. For a fractional $F$-packing $\phi$ of $G$, we let $\partial\phi(e)$ be the total weight of $\psi$ over the edge $e$ of $G$, that is $\partial\phi(e):= \sum_{F\in G: e\subseteq F} \psi(F)$. For an $r$-graph $G$, we let $\mc{K}_q^r(G)$ denote the subgraphs of $G$ isomorphic to $K_q^r$. As explained in the proof outline, an almost decomposition is an assignment of weights such that every edge receives a total weight of roughly $1$, as follows.

\begin{definition}\label{def:AlmostDecompo}
Let $F$ be a hypergraph and let $\eta\in [0,1]$. An \emph{$\eta$-almost fractional $F$-decomposition} of a hypergraph $G$ is a fractional $F$-packing $\phi$ of $G$ such that for every edge $e\in E(G)$, $\partial\phi(e)$ is in $[1-\eta,1]$. 
\end{definition}

\subsection{Concatenation Lemma}

Before diving into our main decomposition results, we prove the following {\em concatenation lemma}. It extends, to more general decompositions, the natural observation that if a graph $G$ admits an $H$-decomposition, and $H$ admits an $H'$-decomposition, then $G$ admits an $H'$-decomposition. Specifically, we need such a lemma to hold for $\eta$-almost fractional decompositions, and for assignments of weights to families $\mc{H}$ of subgraphs of $G$. To that end, we define the following concept of $\mc{H}$-decompositions.

\begin{definition}\label{def:FamilyDecompo}
Let $G$ be a hypergraph and let $\mc{H}$ be a family of subgraphs of $G$. A \emph{fractional $\mc{H}$-packing} of $G$ is an assignment $\phi$ of non-negative weights to elements of $\mc{H}$ such that for each $e\in G$, we have $\partial \phi(e) := \sum_{H\in \mc{H}: e\in H} \phi(H) \in [0,1]$. We say the fractional $\mc{H}$-packing is a \emph{fractional $\mc{H}$-decomposition} if $\partial \phi(e) = 1$ for all $e\in G$. For a real $\eta\in [0,1]$, we say the fractional $\mc{H}$-packing is an \emph{$\eta$-almost fractional $\mc{H}$-decomposition} if $\partial \phi(e) \in [1-\eta, 1]$ for all $e\in G$. 
\end{definition}

A fractional $K_q^r$-decomposition of an $r$-graph $G$ is then equivalent to a fractional $\mc{K}_q^r(G)$-decomposition of $G$; similarly for $\eta\in [0,1]$, an $\eta$-almost fractional $K_q^r$-decomposition of $G$ is equivalent to an $\eta$-almost fractional $\mc{K}_q^r(G)$-decomposition of $G$. 

\begin{lem}[Concatenation]\label{prop:AlmostConcatenates}
Let $G$ be an $r$-graph and let $\mc{H}$ and $\mc{H}'$ be families of subgraphs of $G$. For each $H\in \mc{H}$, let $\mc{H}'(H) := \{H' \in \mc{H}': H' \subseteq H\}$. If there exists $\alpha,\beta\in [0,1]$ such that $G$ has an $\alpha$-almost fractional $\mc{H}$-decomposition and each $H\in \mc{H}$ has a $\beta$-almost fractional $\mc{H}'(H)$-decomposition, then $G$ has an $(\alpha+\beta)$-almost fractional $\mc{H}'$-decomposition.  
\end{lem}
\begin{proof}
    Let $\phi_0$ be the $\alpha$-almost fractional $\mc{H}$-decomposition of $G$, and for every $H\in\mc{H}$, let $\phi_H$ be the $\beta$-almost fractional $\mc{H}'(H)$-decomposition of $H$. We assign the following non-negative weight to every $H'\in\mc{H}'$,
    \[\psi(H'):=\sum_{H\in\mc{H}:H'\subseteq H}\phi_H(H')\cdot\phi_0(H).\]
    For every edge $e\in G$, we have
    \begin{align*}
        \partial\psi(e) 
        = \sum_{H'\in \mc{H}': e\in H'} \psi(H')
        &= \sum_{H'\in \mc{H}': e\in H'} \sum_{H\in\mc{H}:H'\subseteq H}\phi_H(H')\cdot\phi_0(H) \\
        &= \sum_{H\in\mc{H}: e\in H} \left[\phi_0(H)\cdot\sum_{H'\in \mc{H}'(H): e\in H'} \phi_H(H')\right] \\
        &= \sum_{H\in\mc{H}: e\in H} \phi_0(H)\cdot\partial\phi_H(e).
    \end{align*}
    Therefore, for every edge $e\in G$, we have
    \[(1-\beta)\cdot\sum_{H\in\mc{H}: e\in H} \phi_0(H) 
    \leq \partial\psi(e) 
    \leq \sum_{H\in\mc{H}: e\in H} \phi_0(H).\]
    hence,
    \((1-\beta)(1-\alpha) \leq \partial\psi(e) \leq 1,\) and $ \partial\psi(e)\in[1-(\alpha+\beta),1]$, as desired.
\end{proof}

\subsection{Missing One Edge and Fixing}

Our second preliminary result is a ``fixing'' statement and its main corollary: we show that an $\eta$-almost fractional $K_{rq}^r$-decomposition of a hypergraph $G$, for small enough $\eta>0$, implies the existence of a fractional $K_{q}^r$-decomposition of $G$. To do so, we start by showing that $K_{rq}^r-e$ admits a fractional $K_q$-decomposition for every edge $e\in K_{rq}^r$. By taking such a decomposition of $K_{rq}^r-e$ for each edge $e\in K_{rq}^r$, and averaging with some well-chosen weights, we prove a {\em fixing} result; for this, we were inspired by a clever idea of Montgomery from his work on fractional decompositions of dense graphs~\cite{montgomery2019fractional}. Namely, we prove that for any choice of mapping $\phi:E(K_n^r)\to[1-\eta,1]$, there exists a fractional $K_q^r$-packing of $K_{rq}^r$ such that the sum of the weights along each edge is exactly $\phi(e)$. We can then take an $\eta$-almost fractional $K_{rq}^r$-decomposition of a hypergraph $G$, and use our fixing result to find a $K_q^r$-packing of each copy of $K_{rq}^r$ in $G$, whose weights are carefully defined to compensate the missing weights on each edges, hence yielding a fractional $K_{q}^r$-decomposition of $G$.

\begin{thm}\label{thm:MissingOneEdge}
Let $q>r\ge2$ be integers and let $n:=rq$.  If $e\in K_n^r$, then $K_n^r-e$ has a fractional $K_q^r$‑decomposition.
\end{thm}
\begin{proof}
For $i\in [r-1]_0$, let $\mc{Q}_i := \{ Q\in \mc{K}_q^r(K_n^r): |V(Q)\cap V(e)|=i\}$. We seek to define non-negative reals $(w_i: i\in [r-1]_0)$ such that defining $\phi(Q) := w_{|V(Q)\cap V(e)|}$ yields a fractional $K_q^r$-decomposition of $K_n^r-e$. (Note that if a fractional $K_q^r$-decomposition of $K_n^r-e$ exists, then an average over all permutations of $e$ and $[n]-e$ yields such a symmetrized decomposition). 

For $\phi$ to yield a fractional $K_q^r$-decomposition, we require that $\partial\phi(f)=\sum_{Q\supseteq f} \phi(Q) = 1$ for every edge $f\in K_n^r-e$. Now fix $f\in K_n^r-e$ with $|f\cap e|=t\in [r-1]_0$. then every clique $Q$ containing $f$ satisfies $|Q\cap e|=i$ for some $i\ge t$ and to form such a clique one chooses exactly $i-t$ additional vertices from $e\setminus f$ (which is of size $r-t$) and exactly $q-r-(i-t)$ vertices from $V(K_n^r)\setminus (e\cup f)$ (which is of size $n-2r+t$). Thus, the $w_i$ need to satisfy the following equation:
$$\sum_{i\ge t}^{r-1} \binom{r-t}{i-t}\binom{n-2r+t}{q-r-i+t} w_i = 1.$$
Hence we define for all $t,i\in [r-1]_0$ and  $a_{t,i} := \binom{r-t}{i-t}\binom{n-2r+t}{q-r-i+t}$ if $t\le i$ and $a_{t,i}:= 0$ otherwise. Thus, we seek the solution of the system of $r$ equations $Aw=\mathbf{1}$ where $A$ is the matrix with values $a_{t,i}$ and $w$ is the vector with values $w_i$ and $\mathbf{1}$ denotes the all ones vector. 

Note that as $n\ge rq \ge q+r$ (as $q> r\ge 2$), we have that $n-2r \ge q-r$ and so all of the $a_{t,i}$ are well-defined and $a_{t,i} \ge 1$ for all $t \le i$. Hence $A$ is a real upper-triangular matrix with positive diagonal entries and hence has a non-zero determinant and so is invertible. Thus $w=A^{-1}\cdot \mathbf{1}$. 

It remains to show that $w_t \ge 0$ for all $t\in [r-1]_0$. We prove this by reverse induction on $i$. Namely, $w_{r-1} = 1/a_{r-1,r-1} \ge 0$. So we assume by induction that $w_{i}\ge 0$ for all $i > t$ and we seek to prove $w_t \ge 0$. We note that by assumption $\sum_{i\ge t+1}^{r-1} a_{t+1,i}\cdot w_i = 1$ and $\sum_{i\ge t}^{r-1} a_{t,i}\cdot w_i = 1$. 

The key observation to note is that for all $i > t$, we have $a_{t,i} \le a_{t+1,i}$ since 
$$\frac{a_{t,i}}{a_{t+1,i}} = \frac{\binom{r-t}{i-t} \binom{n-2r+t}{q-r-i+t}}{\binom{r-t-1}{i-t-1} \binom{n-2r+t+1}{q-r-i+t+1}} = \frac{r-t}{i-t} \cdot \frac{q-r-i+t+1}{n-2r+t+1} \le r \cdot \frac{q-r}{n-2r} \le 1$$
where we used that $i-t\ge 1$ and that $r(q-r) \le rq-r^2\le rq-2r = n-2r$ since $r\ge 2$. But then we find that
$$a_{t,t} \cdot w_t = 1 - \sum_{i \ge t+1}^{r-1} a_{t,i}\cdot w_i \ge 1 - \sum_{i\ge t+1} a_{t+1,i}\cdot w_i = 1-1 = 0$$
Since $a_{t,t}\ge 1$, we then find that $w_t\ge 0$ as desired.
\end{proof}

By taking the decompositions of $K_n-e$ given by~\cref{thm:MissingOneEdge} for all edges $e\in K_n^r$, and averaging with some well-chosen weights, we prove the following {\em fixing} result, finding a fractional packing $\Phi$ with fixed and pre-determined values for $\partial\Phi$.

\begin{cor}\label{cor:ColorTrick}
Let $q > r \ge 2$ be integers and let $n:=rq$. If $\phi:E(K_n^r)\rightarrow \Big[1-\frac{1}{\binom{n}{r}},1\Big]$, then there exists a fractional $K_q^r$-packing $\Phi$ of $K_n^r$ such that $\partial \Phi(e) = \phi(e)$ for all $e\in K_n^r$. 
\end{cor}
\begin{proof}
Let $\Phi_0$ denote a fractional $K_q^r$-decomposition of $K_n^r$ (such exists by symmetry); for each edge $e\in K_n^r$, let $\Phi_{e}$ denote a fractional $K_q^r$-decomposition of $K_n^r-e$; note that such a decomposition exists by Theorem~\ref{thm:MissingOneEdge}. 

Let $\varepsilon := \frac{1}{\binom{n}{r}}$. For each $e\in K_n^r$, let $\lambda_e := \frac{\phi(e) - (1-\varepsilon)}{\varepsilon}$. Note that since $\phi(e) \in [1-\varepsilon,1]$, we have that $\lambda_e\in [0,1]$. Now we define 
$$\Phi_e' := \lambda_e \cdot \Phi_0 + (1-\lambda_e) \cdot \Phi_e.$$ 
Note that $\Phi'_e$ is fractional $K_q^r$-packing of $K_n^r$ such that $\partial \Phi'_e(e) = \lambda_e$ and $\partial \Phi'_e(f) = 1$ for each $f\in K_n^r-e$. 
Finally we set
$$\Phi := \frac{1}{\binom{n}{r}} \cdot \sum_{e\in K_n^r} \Phi'_e.$$
Note that $\Phi$ is a fractional $K_q^r$-packing of $K_n^r$ (since it is the average of fractional $K_q^r$-packings of $K_n^r$). Furthermore for each $f\in K_n^r$, we have that
\begin{align*}
\partial \Phi(f) = \frac{1}{\binom{n}{r}} \cdot \sum_{e\in K_n^r} \partial \Phi'_e(f) = \frac{1}{\binom{n}{r}} \cdot \left( \binom{n}{r}-1 + \lambda_f\right) = 1 - \varepsilon + \varepsilon \cdot \lambda_f = \phi(f),
\end{align*}
as desired.
\end{proof}

We now prove that we can transform an almost fractional $K_{rq}^r$-decomposition into a fractional $K_q^r$-decomposition.

\begin{cor}\label{cor:AlmostToFull}
Let $q > r \ge 2$ be integers. If an $r$-graph $G$ has a $\frac{1}{\binom{rq}{r}}$-almost fractional $K_{rq}^r$-decomposition, then $G$ has a fractional $K_q^r$-decomposition.     
\end{cor}
\begin{proof}
Let $\varphi$ be a $\frac{1}{\binom{rq}{r}}$-almost fractional $K_{rq}^r$-decomposition of $G$. For each $Q\in \mc{K}_{rq}^r(G)$, let $\phi_Q$ be defined as $\phi_Q(e) := \frac{1-\frac{1}{\binom{rq}{r}}}{\partial \varphi(e)} $ for each $e\in Q$. Note that $\phi_Q(e) \in \left[1-\frac{1}{\binom{rq}{r}},1\right]$ for all $e\in Q$ since $\varphi(e)$ is. By Corollary~\ref{cor:ColorTrick}, there exists a fractional $K_q^r$-packing $\Phi_Q$ of $Q$ such that $\partial \Phi_Q(e)=\phi_Q(e)$ for all $e\in Q$. 

Now for each $Q' \in \mc{K}_q^r(G)$, let 
$$\Phi(Q'):= \sum_{Q\in \mc{K}_{rq}^r(G): Q'\subseteq Q} \varphi(Q)\cdot \Phi_Q(Q').$$
Note that $\Phi(Q')\ge 0$ for all $Q'\in \mc{K}_q^r(G)$ and furthermore for each $e\in G$, we have that
\begin{align*}
\partial \Phi(e) &= \sum_{Q'\in \mc{K}_q^r(G): e\in Q'} \Phi(Q') = \sum_{Q'\in \mc{K}_q^r(G): e\in Q'}~~\sum_{Q\in \mc{K}_{rq}^r(G): Q'\subseteq Q} \varphi(Q)\cdot \Phi_Q(Q')\\
&= \sum_{Q\in \mc{K}_{rq}^r(G): e\in Q}~~\sum_{Q'\in \mc{K}_q^r(Q): e\in Q'} \varphi(Q)\cdot \Phi_Q(Q') = \sum_{Q\in \mc{K}_{rq}^r(G): e\in Q} \varphi(Q)\cdot \partial\Phi_Q(e) \\
&= \sum_{Q\in \mc{K}_{rq}^r(G): e\in Q} \varphi(Q)\cdot \phi_Q(e) = \partial \varphi(e) \cdot \frac{1-\frac{1}{\binom{rq}{r}}}{\partial \varphi(e)} = 1-\frac{1}{\binom{rq}{r}},
\end{align*}
and hence $\frac{1}{1-\frac{1}{\binom{rq}{r}}} \cdot \Phi$ is a fractional $K_q^r$-decomposition of $G$ as desired. 
\end{proof}

\section{Missing Matchings}\label{sec:MissingMatchings}

We now show that, if $M$ is the union of some matchings, at most some ${\rm polylog}(n)$ of them, then $K_n-M$ admits a fractional $K_q^r$-decomposition. We do this in three steps. We first prove that $K_n-M$ admits an almost fractional decomposition if $M$ is a matching. We then apply our fixing result,~\cref{cor:AlmostToFull}, to yield a fractional decomposition, and use induction to obtain the desired result. We first require the following lemma.

\begin{lem}\label{lem:ProbDist}
Let $r\ge 2$ be an integer and let $p\in (0,1/2]$. If $S$ is a set of size $r$, then there exists a probability distribution $\phi$ on the subsets of $S$ such that $\phi(S)=0$ and for each $T\subsetneq S$, we have $\sum_{T': T\subseteq T'\subseteq S} \phi(T') = p^{|T|}$.
\end{lem}

\begin{proof}
For each $T\subsetneq S$, let $f(T)=p^{|T|}$; let $f(S) = 0$. By M\"obius inversion, see e.g.~\cite[Example 3.8.3]{stanley2011enumerative}, $\phi$ satisfies $\sum_{T': T\subseteq T'\subseteq S} \phi(T') = f(T)$ for all $T\subseteq S$ if and only if for all $T\subseteq S$, we set 
$$\phi(T) := \sum_{T':T\subseteq T'\subseteq S} (-1)^{|T'\setminus T|}\cdot f(T') .$$   
For every $T\subseteq S$, with $|T|=t$, we then have
\begin{align*}
\phi(T) 
= \Big[ \sum_{k=0}^{r-t} \binom{r-t}{k}(-1)^{k}\cdot p^{t+k}\Big] - (-1)^{r-t}\cdot p^{r}
= p^{t}\cdot \Big( (1-p)^{r-t} - (-1)^{r-t}p^{r-t}\Big).
\end{align*}

For $p\in (0,1/2]$, we have that $1-p\ge p$ and hence $\phi(T)$ is positive for all $T\subseteq S$ (and indeed also at most $1$). Finally, we note that $f(\emptyset)=1$ and yet $\sum_{T\subseteq S} \phi(T)=f(\emptyset)$ by definition. Combining, we find that $\phi$ is the desired probability distribution.
\end{proof}

We are now ready to prove that the existence of the probability distribution guaranteed by~\cref{lem:ProbDist}, implies our desired result when $M$ is one matching.

\begin{thm}\label{thm:OneMatchingMissing}
For each integer $r\ge 2$, there exists an integer $C\ge 1$ such that the following holds: Let $q$ and $n$ be integers such that $q > r$ and $n\ge Cq$. If $M$ is a matching of $K_n^r$, then $K_n^r-M$ has a fractional $K_q^r$-decomposition.
\end{thm}

\begin{nonlateproof}{thm:OneMatchingMissing}
Let $C\geq 32r^3$. Suppose that $M = \{e_1,\ldots, e_m\}$ and let $p:= \frac{1}{2}$. For each $i\in [m]$, let $\phi_i$ be the probability distribution as guaranteed by~\cref{lem:ProbDist} on the set $S_i:=V(e_i)$ for probability $p$. Select a random set of vertices $X\subseteq V(K_n^r)$ by independently choosing, for each $i\in [m]$, a subset $X_i$ of $S_i$ according to $\phi_i$, and by choosing each vertex of $V(K_n^r)\setminus V(M)$ independently at random with probability $p$. Thus, for each $T\subseteq V(K_n^r)$, letting $T_i := T \cap S_i$ for each $i\in [m]$, 
$$\Prob{T=X} = p^{|T\setminus V(M)|} \cdot \prod_{i\in [m]} \Prob{T_i= X_i}.$$

If there exists an edge $e_i\in M$ such that $V(e_i)$ is contained in $T$, that is,
if $S_i\subseteq T$, then we obtain that $\Prob{X=T} = 0$, since $\phi_i(S_i)=0$ by~\cref{lem:ProbDist}, thus the following claim holds.
\begin{claim}\label{cl:Cliques}
    If $\Prob{X=T} > 0$, then $T$ induces a clique in $K_n^r-M$.
\end{claim}

Assume now that there does not exist such an edge $e_i$. Then, we find that
\[\Prob{T\subseteq X} 
= p^{|T\setminus V(M)|} \cdot  \prod_{i\in [m]} \Prob{T_i\subseteq X_i} 
= p^{|T\setminus V(M)|} \cdot  \prod_{i\in [m]} p^{|T_i|} = p^{|T|},\]
where we used $\Prob{T_i\subseteq X_i} = p^{|T_i|}$ as guaranteed by~\cref{lem:ProbDist}. In particular, for each $e\in K_n^r-M$, we find that $\Prob{V(e)\subseteq X} = p^r$. We now prove that, conditioning on that event, the random set $X$ contains at least $rq$ vertices with sufficiently large probability. 

\begin{claim}\label{cl:Prob}
For each $e\in K_n^r-M$, 
\[\Prob{|X|< rq~|~V(e)\subseteq X} \le  \frac{1}{\binom{rq}{r}}.\]    
\end{claim}
\begin{proofclaim}
Let $I:= \{i\in [m]: e\cap e_i\ne \emptyset\}$, and $$Y:= |\{i\in [m]\setminus I: |X_i|\ge 1\}| + |X\setminus (V(M)\cup V(e))|.$$ 
Thus $Y$ is the sum of independent Bernoulli random variables, and is independent from $\{V(e)\subseteq X\}$. For each $i\in [m]$, let $v_i$ be a vertex in $V(e_i)$, and observe that, by~\cref{lem:ProbDist} we have
\(\Prob{|X_i|\ge 1} \geq \Prob{v_i\subseteq X_i}= p,\)
while for each $v\in V(K_n^r)\setminus V(M)$, we have $\Prob{v\in X} = p$. Therefore, with $m\leq \frac{n}{r}$,
\[\mathbb{E}{Y}\geq p\cdot m + p\cdot (n-mr-r)
\geq p\left(\frac{n}{r}-r\right) \ge \frac{Cq}{2r}-\frac{r}{2} \geq  \frac{Cq}{4r} \geq 2rq\]
where we used $C\ge 8r^2$ and $q\ge r \ge 2$. By the Chernoff bound, we find that
\[\Prob{~Y < rq~} 
\leq \mathbb{P}\Big[~Y < \frac12\cdot~ \frac{Cq}{4r}~\Big]
\leq e^{-(Cq/4r)/8} \le e^{-r^2q} \le \frac{1}{(rq)^r} \le \frac{1}{\binom{rq}{r}},\]
where for the third inequality we used that $C\ge 32r^3$ and for the second to last inequality we used that $rq \ge \ln (rq)$. With $|X|\ge Y$, we conclude that
$$\Prob{~|X| < rq~|~e\subseteq X} \le \Prob{~Y<rq~|~e\subseteq X} = \Prob{~Y<rq~} \le \frac{1}{\binom{rq}{r}},$$
as desired.
\end{proofclaim}
Let $\mc{H}$ be the family of cliques $T\subseteq K_n^r-M$ with $|T|\ge rq$. For every $T\in\mc{H}$, let $\Phi(T) := \frac{\Prob{X=T}}{p^r}$. For every edge $e\in K_n^r-M$, we have
\[\partial\Phi(e)=\sum_{T\in\mc{H}:e\subseteq T}\frac{\Prob{X=T}}{p^r} 
= \frac{\sum_{T\in\mc{H}:e\subseteq T}\Prob{X=T}}{\Prob{V(e)\subseteq X}}  = \Prob{|X|\geq rq~|~V(e)\subseteq X},\]
where the last equality follows from~\cref{cl:Cliques}. By Claim~\ref{cl:Prob}, we obtain that $\Phi$ is a $\frac{1}{\binom{rq}{r}}$-almost fractional $\mc{H}$-decomposition of $K_n^r-M$. By symmetry, every $T\in \mc{H}$ admits a fractional $K_{rq}^r$-decomposition. By Proposition~\ref{prop:AlmostConcatenates}, we can concatenate these decompositions, and we obtain a $\frac{1}{\binom{rq}{r}}$-almost fractional $K_{rq}^r$-decomposition of $K_n^r-M$. By~\cref{cor:AlmostToFull}, there exists a fractional $K_q^r$-decomposition of $K_n^r-M$, as desired.
\end{nonlateproof}

We can finally use an induction step to prove our desired result.

\begin{cor}\label{cor:MissingManyMatchings}
For each integer $r\ge 2$, there exists an integer $C\ge 1$ such that the following holds: Let $q,n$ and $m$ be positive integers such that $q > r$ and $n\ge C^m \cdot q$. If $M$ is the union of $m$ matchings of $K_n^r$, then $K_n^r-M$ has a fractional $K_q^r$-decomposition.
\end{cor}
\begin{proof}
We use the same $C$ as in Theorem~\ref{thm:OneMatchingMissing}. We proceed by induction on $m$. For $m=1$, this is simply Theorem~\ref{thm:OneMatchingMissing}. Let $M:= \{M_1,\ldots, M_m\}$. Let $M':= \{M_1,\ldots ,M_{m-1}\}$. By induction, $K_n^r-M'$ has a fractional $K_{Cq}^r$-decomposition. But by Theorem~\ref{thm:OneMatchingMissing}, we find that for each $Q\in \mc{K}_{Cq}^r(K_n^r-M')$ we have that $Q\setminus M_m$ has a fractional $K_q^r$-decomposition. Hence $K_n^r-M$ has a fractional $K_q^r$-decomposition as desired.  
\end{proof}

\section{A Sampling Lemma and the Proof of Main Result}

\subsection{A Sampling Lemma}\label{sec:sampling}

This section is dedicated to the last main piece of the puzzle before proving~\cref{thm:Main}. We show that if $G$ is a dense enough hypergraph, then we can find an almost fractional $\mc{H}$-decomposition of $G$, where $\mc{H}$ is the family of all subgraphs of $G$ isomorphic to $K_k$ minus $m$ matchings, for some well-chosen parameters $k,m$. For this proof, we need the following proposition.

\begin{proposition}\label{prop:Exploration}
Let $J$ be an $r$-uniform hypergraph, and let $X\subseteq V(J)$ such that $d_J(v)\geq 1$ for every vertex $v\in V(J)\setminus X$. Let $m:= |V(J)\setminus X|$. Then there exists an ordering $v_1,\ldots, v_m$ of $V(J)\setminus X$ such that the number of $i\in [m]$ where $d_{J[X\cup \{v_j: j\in [i]\}]}(v_i)\ge 1$ is at least $\frac{1}{r} |V(J)\setminus X|$.
\end{proposition}
\begin{proof}
Let $S\subset V(J)\setminus X$ and $v_1,\ldots, v_{\ell}$ be an ordering of $S$ such that the number of $i\in [\ell]$ where $d_{J[X\cup \{v_j: j\in [i]\}]}(v_i)\ge 1$ is at least $\frac{1}{r} |S|$, and subject to that $|S|$ is maximized. Note that such a set is well-defined, as the empty set satisfies the conditions. If $S=V(J)\setminus X$, then the ordering is as desired. 

So we assume $V(J)\setminus (X\cup S)\ne \emptyset$. Let $U:= V(J)\setminus (X\cup S)$. Thus, there exists a vertex $u\in U$. As $d_J(u)\geq 1$, there exists a set $S'\subseteq U\setminus\{u\}$ with $|S'|\leq r-1$, such that $u$ has degree at least $1$ in $J[X\cup S \cup S']$. Let $v_{\ell+1},\ldots,v_{\ell+|S'|}$ be an ordering of $S'$ and let $v_{\ell+|S'|+1}:=u$; observe that, by construction, there exists an edge of $J$ containing $u$ whose other vertices precede $u$ in the ordering $v_1,\ldots, v_{\ell+|S'|+1}$. Let $S'':= S\cup S'$. It follows that the number of $i\in [\ell+|S'|+1]$ where $d_{J[X\cup \{v_j: j\in [i]\}]}(v_i)\ge 1$ is at least $1+\frac{1}{r} |S| \ge \frac{1}{r}|S''|$ and hence $S''$ contradicts the maximality of $S$.
\end{proof}

\begin{thm}\label{thm:AlmostBySampling}
For every integer $r \geq 2$ and $m,k\geq r$, let $G$ be an $r$-graph with $\delta(G)\ge (1-d)n$ where $d \leq \frac{1}{(2e^2k)^{r-1}}$. If $\mc{H}$ is the set of induced subgraphs $H$ of $G$ with exactly $k$ vertices such that $\Delta_1(\bar{H})\le m$, then $G$ has a $ k\cdot ((2\cdot e^2\cdot k)^{r-1}\cdot d)^{(m^{1/r-1}-r)/(r-1)}$-almost fractional $\mc{H}$-decomposition.
\end{thm}

\begin{nonlateproof}{thm:AlmostBySampling}
For each $H\in \mc{H}$, we let $\phi(H):= \frac{1}{\binom{n-r}{k-r}}$. We claim that $\phi$ is the desired almost fractional decomposition. Note that $\phi(H)\ge 0$ for all $H\in \mc{H}$. Furthermore, we have that $\partial \phi(f) \le 1$ for every edge $f\in G$ since each edge is in at most $\binom{n-r}{k-r}$ elements of $\mc{H}$. Thus it remains to show that $\partial \phi(f) \ge 1 - k\cdot [(2\cdot e^2\cdot k)^{r-1}\cdot d]^{(m^{1/(r-1)}-r)/(r-1)}$ for each $f\in G$.

To that end, fix an edge $f\in G$. Choose uniformly at random a set $S\subseteq V(G)\setminus V(f)$ of size $k-r$, and let
\(H':=G[\,S\cup V(f)]\). Then $H'$ is a uniformly random induced $k$-vertex supergraph of $f$, and hence
\[
\partial \phi(f)=\Prob{H'\in \mc{H}}\ =\ \Prob{\Delta_1(\bar{H'})\leq m} \geq 1-\sum_{u\in S\cup V(f)}\Prob{d_{\bar{H'}}(u)>m},\]
by the union bound. Thus, it suffices to prove the following.

\begin{claim}\label{cl:Prob3}
For every vertex $u\in S\cup V(f)$,
\[\Prob{d_{\bar{H'}}(u) > m}\leq \big[(2\cdot e^2\cdot k)^{r-1}\cdot d\big]^{(m^{1/(r-1)}-r)/(r-1)}.\]
\end{claim}

\begin{proofclaim}
Fix $u\in S\cup V(f)$, and let $X=V(f)\setminus\{u\}$, hence $|X|=r$ or $|X|=r-1$. Let $J$ be the link graph of $\bar{H}'$ at the vertex $u$, keeping only vertices incident with $u$, that is
\[ V(J)=\{w\in S\cup V(f)\colon~\exists\ e\in \bar{H}',~\{u,w\}\subseteq e\},\qquad
E(J)=\{e\setminus\{u\}\colon e\in \bar{H}',~u\in e\}.
\]
Let $s:=|V(J)\setminus X|$. Observe that $J$ is an $(r-1)$-uniform hypergraph with $d_{\bar{H'}}(u)$ edges, therefore, if $d_{\bar{H'}}(u)>m$, $J$ contains at least $m^{1/(r-1)}$ vertices, and $s\geq m^{1/(r-1)} - r$.

By definition, every vertex $v\in V(J)\setminus X$ satisfies $d_J(v)\geq 1$. Therefore, by~\cref{prop:Exploration}, there exists an ordering $v_1,\ldots,v_s$ of $V(J)\setminus X$ such that the number of $i\in [m]$ where $d_{J[X\cup \{v_j: j\in [i]\}]}(v_i)\ge 1$ is at least $\frac{s}{r-1}$. Call these indices {\em good}.

We count the number of $(k-r)$ sets $S$ that contain such an ordering. For every good index, there are at most $\binom{s}{r-2}dn$ possible vertices, (recall that $\Delta(\bar{G})\leq dn$, and there are at most $\binom{s}{r-2}$ preceding sets of $r-2$ vertices), while for the other indices, there are at most $\binom{n-r}{s-s/(r-1)}$ choices of vertices. Finally, there are at most $\binom{n-r-s}{k-r-s}$ choices of vertices for elements in $S\setminus V(J)$. Therefore, we obtain that the number of $(k-r)$ sets $S$ containing such an ordering is at most

\[N(s):= \binom{\binom{s}{r-2}dn}{\frac{s}{r-1}}\cdot  \binom{n-r}{s-\frac{s}{r-1}}\cdot \binom{n-r-s}{k-r-s},\]
and
\[\mathcal{P}:=\Prob{d_{\bar{H'}}(u) > m}\leq N(s)/\binom{n-r}{k-r},\text{ with }s\geq m^{1/(r-1)}-r.\]
Observe that $s-\frac{s}{r-1}=\frac{s(r-2)}{r-1}$ and that,
\begin{align*}
  \binom{n-r-s}{k-r-s}/\binom{n-r}{k-r}=\frac{(n-r-s)!}{(n-k)!\cdot(k-r-s)!}\cdot\frac{(n-k)!\cdot(k-r)!}{(n-r)!}
  =\frac{(k-r)\ldots(k-r-s+1)}{(n-r)\ldots(n-r-s+1)}\leq \Big(\frac{k}{n}\Big)^s,  
\end{align*}
hence we deduce
\[\mc{P}\leq k^s\cdot n^{-s}\cdot\brackets*{ \binom{\binom{s}{r-2}dn}{\frac{s}{r-1}}\cdot \binom{n-r}{s-\frac{s}{r-1}}}.\]
If $r=2$, we obtain that $\mc{P}\leq k^s\cdot n^{-s}\cdot \binom{dn}{s}\leq k^s\cdot d^{s}$, as desired. Assume now that $r\geq 3$. Using the standard bound $\binom{n}{k}\leq\big(\frac{en}{k}\big)^k$, and $s-\frac{s}{r-1}=\frac{s(r-2)}{r-1}$, we obtain
\[\mc{P}\leq k^s\cdot d^{s/(r-1)}\cdot\brackets*{\Big(\frac{e\cdot (r-1)\cdot\binom{s}{r-2}}{s}\Big)^{s/(r-1)} \cdot \left(\frac{e\cdot(r-1)}{s\cdot(r-2)}\right)^{s(r-2)/(r-1)}},\]
and 
\[\mc{P}\leq k^s\cdot d^{s/(r-1)}\cdot\brackets*{\frac{e^{2r-3}\cdot (r-1)^{r-1}}{s\cdot (r-2)^{2r-4}}}^{s/(r-1)}.\]
As $s>r>2$, we have $\frac{2r-3}{r-1}\leq 2$ and $\frac{2r-4}{r-1}\geq1$, hence
\[\mc{P}
\leq k^s\cdot d^{s/(r-1)}\cdot
  \brackets*{e^{\frac{2r-3}{r-1}}\cdot \frac{r-1}{(r-2)^{\frac{2r-4}{r-1}}}\cdot\frac{1}{s^{1/(r-1)}}}^{s}
\leq k^s\cdot d^{s/(r-1)} \cdot \brackets*{e^2\cdot \frac{r-1}{r-2}}^s
\leq (2\cdot e^2\cdot k)^s\cdot d^{s/(r-1)}.\]
As $d\leq \frac{1}{(2e^2k)^{r-1}}$ and $s\geq m^{1/(r-1)}-r$, we obtain
\[\mc{P}\leq \brackets*{(2\cdot e^2\cdot k)^{r-1}\cdot d}^{(m^{1/(r-1)}-r)/(r-1)},\]
as desired.
\end{proofclaim}
By the union bound, we obtain that, for each $f\in G$,
\[
\partial \phi(f) \geq 1-\sum_{u\in S\cup V(f)}\Prob{d_{\bar{H'}}(u)>m}
\ge 1 - k\cdot [(2\cdot e^2\cdot k)^{r-1}\cdot d]^{(m^{1/(r-1)}-r)/(r-1)},\]
hence $\phi$ is the desired almost fractional $\mc{H}$-decomposition of $G$.
\end{nonlateproof}

\subsection{Proof of Main Result}\label{sec:MainProof}

We are now ready to prove Theorem~\ref{thm:Main}. 

\begin{proof}[Proof of Theorem~\ref{thm:Main}]
Let $C$ be as in Corollary~\ref{cor:MissingManyMatchings}. Let $m=m(r,\varepsilon)$ be an integer such that $m>\big(r+\frac{r^2-1}{\varepsilon}\big)^{r-1}$, and observe that it implies $m^{1/(r-1)}-r\geq (r^2-1)/\varepsilon > r-1$. 

Let $\alpha:=(2\cdot e^2\cdot \beta \cdot C^m\cdot r)^{-(r-1)}$ for some $\beta>1$ to be defined later. Let $k:=C^m\cdot rq$ and $d:=\frac{\alpha}{q^{r-1+\varepsilon}}$. Let $\mc{H}$ be the set of induced subgraphs $H$ of $G$ with exactly $k$ vertices such that $\Delta_1(\bar{H})\le m$. Observe that 
\[d\cdot (2\cdot e^2\cdot k)^{r-1}=\frac{1}{q^{r-1+\varepsilon}}\cdot \frac{1}{(2\cdot e^2\cdot\beta\cdot C^m\cdot r)^{r-1}}\cdot(2\cdot e^2\cdot C^m\cdot rq)^{r-1}=\beta^{-(r-1)}\cdot q^{-\varepsilon}<1.\]
Therefore, by~\cref{thm:AlmostBySampling}, $G$ has a $k\cdot ((2\cdot e^2\cdot k)^{r-1}\cdot d)^{(m^{1/(r-1)}-r)/(r-1)}$-almost fractional $\mc{H}$-decomposition. Observe that
\[k\cdot ((2\cdot e^2\cdot k)^{r-1}\cdot d)^{(m^{1/(r-1)}-r)/(r-1)}=\frac{C^m\cdot r}{\beta^{m^{1/(r-1)}-r}}\cdot q^{1-\varepsilon(m^{1/(r-1)}-r)/(r-1)}.\]
By our choice of $m$ we have \(\varepsilon\cdot(m^{1/(r-1)}-r)/(r-1)>r+1\), hence
\[k\cdot ((2\cdot e^2\cdot k)^{r-1}\cdot d)^{(m^{1/(r-1)}-r)/(r-1)}\leq \frac{C^m\cdot r}{\beta^{m^{1/(r-1)}-r}}\cdot q^{-r}
\leq e^{-r}\cdot q^{-r}\leq\frac{1}{\binom{rq}{r}},\]
where we used the fact that $\beta$ is large enough with respect to $r$ and $m$. Therefore $G$ has a $\frac{1}{\binom{rq}{r}}$-almost fractional $\mc{H}$-decomposition. By~\cref{cor:MissingManyMatchings}, each $H\in \mc{H}$ has a fractional $K_{rq}^r$-decomposition. By~\cref{prop:AlmostConcatenates}, we can concatenate these decompositions, and we obtain that there exists a $\frac{1}{\binom{rq}{r}}$-almost fractional $K_{rq}^r$-decomposition of $G$. By~\cref{cor:AlmostToFull}, this implies the existence of a fractional $K_q^r$-decomposition of $G$, as desired.
\end{proof}

\bibliographystyle{plain}
\bibliography{bibliography}

\end{document}